\documentclass[11pt, reqno]{amsart}
\usepackage{psfrag}
\usepackage{setspace}
\usepackage{paralist}

\usepackage{multicol}
\usepackage{color}
\usepackage{array,amscd,amsmath,amsthm,ifthen}
\usepackage{amssymb}
\usepackage{graphicx}
\usepackage[all]{xy}
\usepackage{multicol}
\usepackage{cite}
\usepackage{amssymb,amsmath,amscd,graphicx,
latexsym,amsthm}
\usepackage{amssymb,latexsym,amsmath,amscd,graphicx,ifthen,array,multicol,cite,xcolor}
\usepackage[mathscr]{eucal}

\usepackage{scalerel,stackengine}
\DeclareMathSymbol{\shortminus}{\mathbin}{AMSa}{"39}

\stackMath
\newcommand\reallywidehat[1]{%
\savestack{\tmpbox}{\stretchto{%
  \scaleto{%
    \scalerel*[\widthof{\ensuremath{#1}}]{\kern.1pt\mathchar"0362\kern.1pt}%
    {\rule{0ex}{\textheight}}
  }{\textheight}%
}{2.4ex}}%
\stackon[-6.9pt]{#1}{\tmpbox}%
}
 \input xy
 \xyoption{all}
\def\OO{{\mathcal O}}
\def\R{{\mathbf R}}

\def\F{\mathcal{F}}

\def\G{\mathcal{G}}

\def\I{\mathcal{I}}
\def\J{\mathcal{J}}

\def\cP{\mathcal{P}}
\def\Pic0{{\rm Pic}^0}

\def\Aut0{{\rm Aut}^0}

\def\R{{\mathbf{R}}}

\def\*{{\underline *}}

\def\mk{\mathfrak}
\def\b{\mk{b}}
\def\a{\mk{a}}
\def\J{\mathcal{J}}

\DeclareMathOperator{\codim}{codim}

\DeclareSymbolFont{yhlargesymbols}{OMX}{yhex}{m}{n}
\DeclareMathAccent{\yhwidehat}{\mathord}{yhlargesymbols}{"62}

\theoremstyle{plain}
\newtheorem*{introtheorem}{Theorem}
\newtheorem*{introcorollary}{Corollary}
\newtheorem{theorem}{Theorem}[subsection]

\newtheorem{proposition/example}[theorem]{Proposition/Example}

\newtheorem{corollary}[theorem]{Corollary}
\newtheorem{lemma}[theorem]{Lemma}

\theoremstyle{definition}

\newtheorem{conjecture/question}[theorem]{Conjecture/Question}

\newtheorem{remark/definition}[theorem]{Remark/Definition}
\newtheorem{notation/assumptions}[theorem]{Assumptions/Notation}
\newtheorem{setting/notation}[theorem]{Setting and Notation}

\numberwithin{equation}{section}

 \theoremstyle{remark}

\begin{document}

\title[Singularities of  base loci on abelian varieties]{Singularities of  base loci on abelian varieties}

 \author[G. Pareschi]{Giuseppe Pareschi}

\address{Department of Mathematics,  University of Rome Tor Vergata\\Italy}
\email{pareschi@mat.uniroma2.it}
 \thanks{
 The author is partially supported by  the MIUR Excellence Department Project MatMod@TOV awarded to the Department of Mathematics of the University of Rome Tor Vergata, and by PRIN 2022 "Moduli spaces and Birational Geometry". He is a member of GNSAGA - INDAM }

\begin{abstract} We prove that  the log canonical threshold  of the base ideal of a complete linear system on a complex abelian variety is $\ge 1$, and that equality holds if and only if the base locus has divisorial components. Consequently the same assertions hold for the ideal of the intersection of translates of theta divisors by the points of a finite subgroup.
\end{abstract}

\maketitle

A well known theorem of Koll\'ar (\!\!\cite[Theorem 17.13]{kol1}) asserts that for a complex principally polarized abelian variety  the pair $(A,\Theta)$ is log-canonical. This simple and fundamental result has been subsequently sharpened by Ein and Lazarsfeld in \cite{el}, and there are still important  open conjectures in this field, see e.g.  \cite{cgs}, \cite[\S H.29]{mp}, \cite[\S 9]{sy} and references therein. In turn the theorems of Koll\'ar and Ein-Lazarsfeld have been extended to other integral and/or $\mathbb Q$-divisors on abelian varieties  in  \cite{hac1},\cite{hac2},\cite{dh},\cite{p}, \cite{b}, \cite{jl}. 

In this note Koll\'ar's result is extended in a different direction: we are looking at base loci of complete linear systems, and intersections of translates of theta divisors by the points belonging to a finite subgroup. 
We work over $\mathbb C$. For an effective line bundle $L$ on an abelian variety $A$ we denote $\b_L$ the base ideal sheaf of the complete linear system $|L|$, namely the image of the twisted evaluation map $H^0(A,L)\otimes L^{-1}\rightarrow \OO_A$. This is our result

\begin{introtheorem}\label{th:a}
 $lct(\b_L)\ge 1$ and equality holds if and only if the zero locus of $\b_L$ has  divisorial components.
\end{introtheorem}

\noindent (When there are divisorial components the situation is completely understood, see the last part of the proof below). By a standard argument it follows that the conclusion of Theorem holds also for the (scheme-theoretic) intersection of translates of theta divisors of a principally polarized abelian variety $A$ by all points of a finite subgroup $H\subset A$.

 Here and in the sequel, given a morphism $f:X\rightarrow Y$ and an ideal sheaf $\a$ on $Y$ we will denote $f^{-1}\a$ the inverse image ideal sheaf on $X$, i.e. the ideal sheaf generated by the pullbacks of local sections of $\a$. The translation by a point $x\in A$ is denoted $t_x$. 

\begin{introcorollary} Let $A$ be a principally polarized abelian variety, $\Theta$ a theta divisor, and $H$ a finite subgroup of $A$. Let \[ \a_{\Theta,H}=\sum_{h\in H}t_h^{-1}\OO_A(-\Theta)\subset\OO_A.\]
 Then 
 \[
 lct(\a_{\Theta,H})\ge 1
 \]
  and equality holds if and only if the zero locus of $\a_{\Theta,H}$ has divisorial components.
\end{introcorollary}
  When there are no divisorial components  we don't expect  our results to be  sharp or even close to it, and a better understanding of  the log canonical threshold (and more generally of the jumping numbers) of base ideals on abelian varieties seems to be an interesting and difficult problem. 

The motivation for this question comes from the attempt to show that $\b_L$ is a radical ideal sheaf, i.e. the base scheme of $L$ is reduced, or at least, as conjectured by Debarre in \cite[Conjecture 2]{d}, that the 2-codimensional components of the base scheme are generically reduced (see the erratum in \cite{acp}). Although our result does not imply this, we hope that the method we use will be helpful in future developments.

The result of this paper was obtained while I was working with Enrico Arbarello and Giulio Codogni on  the reducedness of the base scheme of line bundles on abelian varieties, and on the other topics of the erratum in \cite{acp}. I thank  them for their generosity, and for valuable conversations and suggestions. More recently, I have greatly benefitted from a continuous exchange with Arbarello. 
 Finally, I thank the referee for many suggestions allowing a better exposition.


\section{Background material on the Fourier-Mukai transform and generic vanishing}

\subsection{Notation} Given an abelian variety $A$ we set $g=\dim A$. Moreover we set $\widehat A=\Pic0 A$ 
 and $\cP$ the Poincar\'e line bundle on $A\times \widehat A$. Given a point $\alpha\in\widehat A$ we denote by $P_\alpha$ the corresponding line bundle on $A$. All sheaves appearing in the sequel are \emph{coherent} sheaves.

Given a sheaf $\F$ on $A$  its cohomological support loci are defined by
\[
V^i(A,\F)=\{\alpha\in\Pic0 A\>|\>h^i(A,\F\otimes P_\alpha)>0\}.
\]
\subsection{(Symmetric) Fourier-Mukai transform associated to the Poincar\'e line bundle}  Denoting $D(X)$ the bounded derived category of the category of coherent sheaves on a projective variety $X$, the Fourier-Mukai functor $\Phi^{A\rightarrow\widehat A}_{\cP}: D(A)\rightarrow D(\widehat A)$ (introduced in the seminal paper \cite{mukai}) is an exact equivalence. Its quasi-inverse is $\Phi^{\widehat A\rightarrow A}_{\cP^\vee[g]}$. 

For our purposes a version of this equivalence, introduced by Schnell in \cite{schnell}, will be more practical.  The \emph{symmetric} Fourier-Mukai transform is defined as the contravariant functor
\[
FM_A:D(A)\rightarrow D(\widehat A)^{op}, \quad FM_A=\Phi^{A\rightarrow\widehat A}_{\cP}\circ \Delta_A
\]
where $\Delta_A: D(A)\rightarrow D(A)^{op}$ is the dualizing functor $\Delta_A=\R\mathcal Hom(\>\cdot\>,\OO_A[g])$. Also $FM_A$ is an exact equivalence, with quasi inverse $FM_{\widehat A}$ (\!\!\cite[Theorem 1.1]{schnell}). 

Let $\pi:A\rightarrow B$ be a surjective homomorphism with connected fibers of abelian varieties and let $i:\widehat B\rightarrow \widehat A$ be the inclusion dual to $\pi$. By \cite[Proposition 1.1]{schnell} the following properties hold: 
\begin{equation}\label{1}FM_A\circ \pi^*=i_*\circ FM_B
\end{equation}
\begin{equation}\label{2}
FM_B\circ  L i^*=R\pi_*\circ FM_A .
\end{equation}
 \subsection{Vanishing conditions}\label{subs:van} Following \cite{schnell}, we have the following vanishing conditions on sheaves on abelian varieties.

 \noindent (a)   A sheaf $\F$ on $A$ is said to be a \emph{GV sheaf} (i.e. satisfies generic vanishing) if $FM_A(\F)$ is a sheaf (in degree zero). If this is the case we denote
\begin{equation}\label{3}
FM_A(\F)=\widehat \F
\end{equation}
It follows if $\F$ is a GV sheaf on $A$ then also $\widehat \F$ is a GV-sheaf on $\widehat A$ and $\widehat{\widehat \F}= \F$. 

If $\F$ is a GV sheaf then the fibre of $\widehat F$ at a point $\alpha\in\widehat A$ is $H^0(A,\F \otimes P_{-\alpha})^\vee$ (\!\!\cite[(10)]{schnell}. It follows that the (reduced structure of the) support  of $\widehat \F$ is the subvariety $-V^0(A,\F)$.

\noindent (b) $\F$ is said to be \emph{M-regular} if $\F$ is GV and $\widehat \F$ is a torsion free sheaf. Since the support of $\widehat\F$ is $-V^0(A,\F)$, if $\F$ is M-regular then $V^0(A,\F)=\widehat A$, i.e. $h^0(A,F\otimes P_\alpha)>0$, for all $\alpha\in \widehat A$.

\noindent (c) $\F$  is said to be \emph{an IT0  sheaf} (i.e. satisfies the index theorem with index 0)  if $\F$ is GV and $\widehat \F$ is a locally free sheaf.

Although not directly used in the sequel, it is perhaps useful to keep in mind the following well known characterizations of the above conditions in terms of the cohomological support loci (which are more commonly adopted as  definitions in the literature):
 
 \noindent  \emph{A coherent sheaf $\F$ is GV (respectively M-regular, IT0) if and only if $\codim_{\widehat A}V^i(A,\F)\ge i$  (resp. $>i$, empty) for all $i>0$. } 

\noindent For the GV condition this is \cite[Corollary 3.12]{gv} (but  one direction  was proved earlier in \cite{hac}). For the M-regularity condition this is \cite[Proposition 2.8]{reg3} and for the IT(0) condition in one direction it follows easily from base change and in the other one it follows e.g. from \cite[Remark 3.13]{gv}.

\subsection{Chen-Jiang decompositions and theorem}\label{subs:cy} Chen-Jiang decompositions were introduced by Z. Jiang and J. Chen in the paper \cite[Theorem 1.1]{cj}.  
A sheaf $\F$ on an abelian variety $A$ is said \emph{to admit a Chen-Jiang decomposition}
if 
\[
\F\cong \bigoplus_i(\pi_{B_i}^*\F_i)\otimes P_{\alpha_i}
\]
where: $\pi_{b_i}:A\rightarrow B_i$ are surjective homomorphisms with connected fibres, $\F_i$ are M-regular sheaves on $B_i$, and $\alpha_i\in \widehat A$ are points of finite order.  

Some remarks about this notion: 

\noindent (a) Chen-Jiang decomposition are essentially unique (\!\!\cite[Lemma 3.5] {lps} and comments after Proposition 3.3, or \cite[\S9]{clp}). 

\noindent  (b) Unless $\dim B_i=\dim A$ (i.e. the homomorphism with connected fibres $\pi_i$ is an isomorphism),  the pullbacks  $\pi_{B_i}^*\F_i$ of  the M-regular sheaves  $\F_i$ are GV but not M-regular (in fact by (\ref{1}) $FM_A(\pi_{B_i}^*\F_i)=i_{\widehat{B_i}*}\widehat{F_i}$ is a torsion sheaf). 

\noindent (c) Therefore a sheaf $\F$ admitting a Chen-Jiang decomposition is always GV, but it is M-regular if and only if the decomposition has $\F$ itself as the only  summand. 

\noindent (d) A direct summand of a sheaf admitting a Chen-Jiang decomposition  admits a Chen-Jiang decomposition (\!\!\cite[Proposition 3.6]{schnell}). 

The theorem of Chen-Jiang \cite[Theorem 1.1]{cj} (together with the subsequent generalization  in \cite{pps}) summarizes and strengthens a good deal of the generic vanishing theorems of Green-Lazarsfeld and Hacon (\!\!\cite{gl1},\cite{gl2}), \cite{hac}). It states that: 

\noindent \emph{If $f:X\rightarrow A$ is a generically finite morphism from a smooth variety to an abelian variety then the sheaf $f_*(\omega_X)$ admits a Chen-Jiang decomposition. }\\


\section{Proofs}
 \subsection{A preliminary lemma. }  Let $L$ be an effective line bundle on an abelian variety $A$. Let $K(L)$ be the subgroup  $\{x\in A\>|\>t_x^*L\cong L\}$, acting on $A$ by translations. The line bundle $L$ is  ample if and only if $K(L)$ is a finite group. 
 
 \begin{lemma}\label{lemma} Let $L$ be an ample line bundle on an abelian variety $A$.  Let $\mathcal I\subset \OO_A$ be a $K(L)$-invariant ideal sheaf, i.e. $t_x^{-1}\I= \I$ for all $x\in K(L)$. Then, for $y\in A$, either $H^0(A,t_y^*L\otimes \mathcal I)=H^0(A,t_y^*L)$ or $H^0(A,t_y^*L\otimes \mathcal I)=0$. Moreover, if $\I$ is non-trivial then the Zariski open set of $y\in A$ such that $H^0(A,t_y^*L\otimes \mathcal I)=0$ is non-empty.
 \end{lemma}
 \begin{proof} Let us briefly recall, somewhat informally,  the following fundamental construction introduced in \cite{mumford} (we refer to the Mumford's paper for the appropriate invariant treatment). Let $\mathcal G(L)$ be the \emph{theta-group} of $L$, namely the group of pairs $(x,\varphi)$, where $x\in K(L)$ and $\varphi: L\rightarrow t_x^*L$ is an isomorphism. This theta-group sits in a central extension
 \begin{equation}\label{eq:0}
 0\rightarrow \mathbb C^* \rightarrow \mathcal G(L)\rightarrow K(L)\rightarrow 0.
 \end{equation}
 The group $K(L)$ does not act naturally on the space of global sections $H^0(A,L)$ but $\mathcal G(L)$ does, in such a way that $\mathbb C^*$ acts with weight $1$, i.e. by homoteties. Mumford's theorem (\!\!\cite[\S1, Proposition 3]{mumford}) states that, up to isomorphism, this is the only irreducible representation of $\mathcal G(L)$ where $\mathbb C^*$ acts with weight $1$. Of course   $K(t_y^*L)= K(L)$ for all points $y\in A$, and the extensions (\ref{eq:0}) for $L$ and $t_y^*L$ are naturally isomorphic, as well as the representations $H^0(A,L)$ and $H^0(A,t_y^*L)$. If $\I$ is a $K(L)$-invariant ideal sheaf then, for all $y\in A$, $H^0(A,t_y^*L\otimes \I)$  is a representation of $\mathcal G(L)$ where $\mathbb C^*$ acts with weight one, and $h^0(A,t_y^*L\otimes\I)\le h^0(A,L)$ . Therefore by Mumford's theorem either $h^0(A,t_y^*L\otimes\I)$ equals $ h^0(A,L)$ or it is zero. If $\I$ is non-trivial then $h^0(A,t_y^*L\otimes\I)$ must be zero for some $y$, otherwise the zero scheme of $\I$ would be contained in the base locus of all translates of $L$. 
 \end{proof}
 
 \subsection{Proof of the Theorem} Given a rational number $c>0$ we denote $\mathcal J(\b_L^c)$ the multiplier ideal sheaf associated to $c$ and $\b_L$. We refer to \cite{laz2} for generalities about multiplier ideal sheaves. 

The Theorem \ref{th:a} states that $\mathcal J(\b_L^c)$ is trivial as soon as $c<1$, and that $\mathcal J(\b_L)$ is trivial if and only if the zero locus of $\b$ has no divisorial components. 

In order to prove this,  in the first place we can assume that the effective line bundle $L$ is ample. Indeed otherwise, as it is well known, there is a surjective homomorphism with connected fibers $\pi:A\rightarrow B$ and an ample line bundle $M$ on  $B$ such that $L=\pi^*M$ (one takes $B=A/K_0(L)$, where $K_0(L)$ is the neutral component of $K(L)$).   Therefore $H^0(A,L)=\pi^*H^0(B,M)$, hence $\b_L=\pi^{-1}\b_M$. Thus $\J(\b_L^c))=
\pi^{-1}\J(\b_M^c)$ (\!\!\cite[Example 9.5.45]{laz2}).

The proof of the first assertion of the Theorem is standard. Let $L$ be an ample line bundle and $c<1$. By Nadel vanishing $V^i(L\otimes \mathcal J(\b_L^c))=\emptyset$ for $i>0$. Therefore the sheaf $L\otimes \mathcal J(\b_L^c)$ is IT0 and the transform $\yhwidehat{L\otimes \mathcal J(\b_L^c)}$ is locally free of rank
\[
h^0(A,L\otimes \mathcal J(\b_L^c)\otimes P_\alpha):= h^0.
\]
Let us recall that, for all $x\in A$, \ $t_x^*L\cong L\otimes P_\alpha$ for some $\alpha\in\widehat A$. If $\mathcal J(\b_L^c)$ were not trivial then $h^0=0$, because the zero locus of $\mathcal J(\b_L^c)$ cannot not be contained in the base locus of all translates of $L$. Therefore $\yhwidehat{L\otimes \mathcal J(\b_L^c)}=0$, hence $L\otimes \mathcal J(\b_L^c)=0$. Since this is impossible
$\mathcal J(\b_L^c)$ is trivial.

\smallskip Let us turn to the second assertion of the Theorem. Assume that $\mathcal J(\b_L)$ is non-trivial. We will prove that  the zero locus of $\b_L$ has divisorial components. To begin with, we recall  that: \emph{ There exists a smooth projective variety $X$, and a generically finite morphism $f: X\to A$ such that 
$L\otimes \J(\b_L)$ is a direct summmand of $f_*\omega_X$. } Let us review this. In fact 
\[\mathcal J(\b_L)=\mathcal J\left(\frac{1}{k}E\right)
\]
where $k>1$  and  $E=E_1+\cdots+E_k$, where the $E_i$'s are general divisors in $|L|$, see \cite[Prop. 9.2.22, and Prop. 9.2.26]{laz2}. Thus $E$ is a divisor in $|kL|$  and it follows from the work of  Esnault-Viehweg (\!\!\cite[(3.13)]{ev}) that
 the sheaf $\mathcal J(\frac{1}{k}E)\otimes L$ is a direct summand of the pushforward of a canonical bundle. A very readable account of this is  \cite[{\it Case $t=1$} p.226]{dh}.
 
 Thus, by the theorem of Chen-Jiang  plus Remark (d) of Subsection \ref{subs:cy}, the (indecomposable!) sheaf $\mathcal J(\frac{1}{k}E)\otimes L$ has a Chen-Jiang decomposition. Therefore:  
there exist a surjective homomorphism of abelian varieties $\pi_{ B}:A\rightarrow B$ with connected fibers, a M-regular sheaf $\G$ on $B$,  and a finite order point $\beta\in \widehat A$ such that
 \begin{equation}\label{eq:1} 
 \mathcal J(\b_L)\otimes L\cong P_\beta\otimes \pi_{ B}^*\G.
 \end{equation}
 The base ideal $\b_L$ is $K(L)$-invariant. Hence also the multiplier ideal $\J(\b_L)$ is $K(L)$-invariant, because, according \cite[\S 3.3]{kol2} one can take a $K(L)$-invariant log resolution of $\b_L$. Therefore, by Lemma \ref{lemma}, $H^0(A, \mathcal J(\b_L)\otimes L\otimes P_\alpha)=0$ for $\alpha$ general in $\widehat A$. Hence the GV sheaf $J(\b_L)\otimes L$, is not M-regular (item (b) of Subsection \ref{subs:van}). Thus $\dim B<\dim A$ in (\ref{eq:1}). 
 
 Since $\G$ is a M-regular sheaf on $B$ we have that $V^0(B,\F)=\widehat B$, and therefore it follows from (\ref{eq:1}) and projection formula that 
 \[V^0(A,L\otimes \J(\b))=i_{\widehat B}(\widehat B)-\beta
 \]
 where 
 $
 i_{\widehat B}:\widehat B\rightarrow \widehat A
 $
  is the inclusion dual to the surjective homomorphism $\pi_{B}$.  
 In (\ref{eq:1}) one can assume $\beta=\hat 0$. Indeed, since $\b_L\subset\J(\b_L)$, $H^0(A,L\otimes\J(\b_L))=H^0(A,L)$, and therefore $\hat 0\in V^0(A,L\otimes \J(\b))= i_{\widehat B}(\widehat B)-\beta$ i.e. $\beta\in i_{\widehat B}(\widehat B)$. Hence we can assume that (\ref{eq:1}) takes the form
  \begin{equation}\label{eq:2} 
 \J(\b_L)\otimes L\cong  \pi_{B}^*\G.
 \end{equation}
 Next, we claim that
   \begin{equation}\label{eq:7} 
 \G\cong {\pi_B}_*L \>.
  \end{equation}
  To prove this,  we apply the symmetric Fourier-Mukai transform to (\ref{eq:2}). Using (\ref{1}) we have
      \begin{equation}\label{eq:5} 
  \yhwidehat{ \J(\b_L)\otimes L}={i_{\widehat B}}_*\widehat \G
  \end{equation} 
    (where $\widehat \G$ is understood to be $FM_B(\G)$). On the other hand, we claim that from Lemma \ref{lemma} it follows that
 \begin{equation}\label{eq:3} 
  \yhwidehat{ \J(\b_L)\otimes L}={i_{\widehat B}}_*i_{\widehat B}^*\widehat L \>.
  \end{equation}
   Indeed applying the functor $FM_A$ to the inclusion $\J(\mk{b})\otimes L\hookrightarrow L$ one gets a surjection  
    \begin{equation}\label{eq:4} 
   \widehat{ L}\rightarrow \yhwidehat{ \J(\mk{b})\otimes L}\rightarrow  0.
 \end{equation}
 As recalled in Subsection \ref{subs:van}, the fiber  $\yhwidehat{ \J(\b_L)\otimes L }(\alpha)= \widehat G(\alpha)$ over a point  $\alpha\in \widehat B$ is identified    to the vector space $
     H^0(A, L\otimes\J(\b_L)\otimes P_{\alpha})^\vee
    $ which, by Lemma \ref{lemma}, is equal 
to $H^0(A, L\otimes P_{\alpha})^\vee$, i.e. the fiber of $\widehat{L }(\alpha)$ over the point  $\alpha$. By Nakayama's lemma it follows that
the surjection  (\ref{eq:4}), restricted to $i_{\widehat B}(\widehat B)$, is the identity. This proves (\ref{eq:3}).
 
 Comparing (\ref{eq:3}) and (\ref{eq:5}) we get
   \begin{equation}\label{eq:6} 
 \widehat \G\cong i_{\widehat B}^*\widehat L
  \end{equation}
Finally, applying the functor $FM_{\widehat A}$ to (\ref{eq:6}) and using (\ref{2}) we get  what claimed, i.e. (\ref{eq:7}). 

  Finally, plugging (\ref{eq:7}) into (\ref{eq:2}) we get
  \[L\otimes \J(\b_L)\cong \pi_B^*{\pi_B}_*L
  \]
 Since $\pi_B^*{\pi_B}_*L$ is locally free, $\J(\b_L)$ must be a line bundle. Since $\J(\b_L)$ is assumed to be non-trivial, $\b_L$ has some divisorial components. This proves one direction of the second assertion of the Theorem.
  
  Conversely, if $L$ is ample and $\b_L$ has some divisorial components it is known that $A\cong B\times C$, where $B$ is a principally polarized abelian variety, $C$ is an abelian variety, and  $L=\OO_B(\Theta_B)\boxtimes L_C$, where $L_C$ is an ample line bundle on $C$ without base divisors (\!\!\cite[Theorem 4.3.1]{bl}). It follows that $\b_L=\OO_B(-\Theta_B)\boxtimes \b_{L_C}$. Hence, by the above, $\J(\b_L)=p_B^*\OO_B(-\Theta_B)$.  This concludes the proof of the Theorem.
  
  \subsection{Proof of the Corollary} We prove more generally the following
  \begin{corollary}\label{cor} Let $L$ be an effective line bundle on an abelian variety and let $H$ be a finite subgroup of $A$. Then 
  \[
  lct \left(\sum_{h\in H}t_h^{-1}\b_{L}\right)\ge 1
  \]
   and equality holds if and only if the zero scheme of $\sum_{h\in H}t_h^{-1}\b_{L}$ has some divisorial components.
  \end{corollary}
  \begin{proof} The argument is standard and similar to \cite[Corollary 8]{acp}. Let $\varphi_L: A\rightarrow \Pic0 A$ be the homomorphism associated to $L$ and let $K:=\varphi_L(H)$. Let us consider the isogeny 
  \[
  f: \Pic0 A\rightarrow (\Pic0 A)/K
  \]
   and let 
   \[
   \pi_K:B\rightarrow A
   \]
    be the dual isogeny. Then 
    \[H^0(B,\pi_K^*L)\cong\bigoplus_{\alpha\in K}\pi_K^*H^0(A,L\otimes P_\alpha).
    \]
  Therefore  \[\b_{\pi_K^*L}=\pi_K^{-1}\left(\sum_{\alpha\in K}\b_{L\otimes P_\alpha}\right)=\pi_K^{-1}\left(\sum_{h\in H}t_h^{-1}\b_L\right).
     \] 
     Corollary \ref{cor} follows from the Theorem because  $\J\!\left(\b_{\pi_K^*L}\right)=\pi_K^{-1}\J\!\left(\sum_{h\in H}t_h^{-1}\b_L\right)$    (\!\!\cite[Example 9.5.44]{laz2}).  \end{proof}

\providecommand{\bysame}{\leavevmode\hbox
to3em{\hrulefill}\thinspace}

\end{document}